\documentclass[11pt]{article}
\usepackage{amsmath,amssymb}

\newtheorem{propo}{{\bf Proposition}}[section]
\newtheorem{coro}[propo]{{\bf Corollary}}
\newtheorem{lemma}[propo]{{\bf Lemma}} \newtheorem{theor}[propo]{{\bf
Theorem}} 

\newenvironment{proof}{{\bf Proof.}}{$\Box$}

\def\N{{\mathbb N}}

\begin{document}

\vspace*{1.0in}

\begin{center} ON MAXIMAL SUBALGEBRAS OF LIE ALGEBRAS CONTAINING ENGEL SUBALGEBRAS 
\end{center}
\bigskip

\begin{center} DAVID A. TOWERS 
\end{center}
\bigskip

\begin{center} Department of Mathematics and Statistics

Lancaster University

Lancaster LA1 4YF

England

d.towers@lancaster.ac.uk 
\end{center}
\bigskip

\begin{abstract}
Relationships between certain properties of maximal subalgebras of a Lie algebra $L$ and the structure of $L$ itself have been studied by a number of authors. Amongst the maximal subalgebras, however, some exert a greater influence on particular results than others. Here we study properties of those maximal subalgebras that contain Engel subalgebras, and of those that also have codimension greater than one in $L$.
\end{abstract}

\section{Introduction}
Relationships between certain properties of maximal subalgebras of a Lie algebra $L$ and the structure of $L$ itself have been studied by a number of authors. Amongst the maximal subalgebras, however, some exert a greater influence on particular results than others. Here we study properties of those maximal subalgebras that contain Engel subalgebras. This idea is somewhat akin to that of maximal subgroups containing Sylow subgroups as introduced by Bhattacharya and Mukherjee in \cite{mbhat}, and further studied in \cite{muck}, \cite{mb} and \cite{wang}.
\par

For $x \in L$, the {\em Engel subalgebra}, $E_L(x)$, is the Fitting null-component relative to ad\,$x$. If $U$ is a subalgebra of $L$, the {\em core} of $U$, $U_L$, is the largest ideal of $L$ contained in $U$. We put 
\[ \mathcal{G} = \{ M : M \hbox{ is a maximal subalgebra of } L \hbox{ and } E_L(x) \subseteq M \hbox{ for some } x \in L \},
\]
\[ G(L) = \bigcap_{M \in \mathcal{G}} M \hbox{ if } \mathcal{G} \hbox{ is non-empty}; \hspace{.2cm} G(L) = L \hbox{ otherwise};\hspace{.2cm} \gamma(L) = G(L)_L.
\]

In section two we consider the structure of $\gamma(L)$ and its relationship to properties of $L$ itself. It is shown that $\gamma(L)$ is nil on $L$ and that $L$ is nilpotent precisely when $\gamma(L) = L$. When $L$ is solvable, then $\gamma(L) = \tau(L) = T(L)_L$, where $T(L)$ is the intersection of the self-idealising maximal subalgebras of $L$ (see \cite{arch}). Necessary and sufficient conditions are found for all maximal subalgebras of $L$ to belong to $\mathcal{G}$, provided that the underlying field is algebraically closed and of characteristic zero.
\par

In section three, necessary and sufficient conditions are found for every maximal subalgebra of $L$ belonging to $\mathcal{G}$ to have codimension one in $L$. This result generalises \cite[Theorem 1]{codone}. Now put
\[ \mathcal{H} = \{M : M \hbox{ is a subalgebra of } L \hbox{ of codimension } > 1 \hbox{ in } L \}, \]
\[ H(L) = \bigcap_{M \in \mathcal{H}} M \hbox{ if } \mathcal{H} \hbox{ is non-empty}; \hspace{.2cm} H(L) = L \hbox{ otherwise};\hspace{.2cm} \eta(L) = H(L)_L, \]
\[ \mathcal{D} = \{M \in {\mathcal G} : M \in \mathcal{H} \}, \]
\[ D(L) = \bigcap_{M \in \mathcal{D}} M \hbox{ if } \mathcal{D} \hbox{ is non-empty}; \hspace{.2cm} D(L) = L \hbox{ otherwise};\hspace{.2cm} \delta(L) = D(L)_L.
\]

We then investigate some basic proerties of $\eta(L)$ and $\delta(L)$, showing, in particular, that if $L$ is solvable then they are supersolvable. Finally it is shown that all maximal subalgebras $M$ of $L$ with $M \in \mathcal{D}$ are c-ideals of $L$ if and only if $L$ is solvable. This generalises \cite[Theorem 3.1]{cideal}.  
\par
Throughout $L$ will denote a finite-dimensional Lie algebra over a field $F$, which is arbitrary unless restrictions on $F$ are specified. We define the {\em derived series} for $L$ inductively by $L^{(0)} = L$, $L^{(i+1)} = [L^{(i)}, L^{(i)}]$ for all $i \geq 0$. The symbol $\oplus$ will denote a direct sum of the underlying vector space structure.  

\section{Maximals containing Engel subalgebras}
First we note that the maximals in $\mathcal{G}$ are precisely those containing Cartan subalgebras, provided that the field has enough elements; also that $G(L)$ is an ideal when $F$ has characteristic zero.

\begin{lemma}\label{l:cart} Let $L$ be a Lie algebra over an field $F$ with at least dim\,$L$ elements. Then $M \in \mathcal{G}$ if and only if $M$ contains a Cartan subalgebra of $L$.
\end{lemma}
\begin{proof} Simply note that the minimal Engel subalgebras are precisely the Cartan subalgebras of $L$, by \cite[Theorem 1]{cart}.
\end{proof}

\begin{lemma}\label{l:ideal} Let $L$ be a Lie algebra over a field $F$ of characteristic zero. Then $G(L) = \gamma(L)$.
\end{lemma}
\begin{proof} For every automorphism $\theta$ of $L$, and every $x \in L$, $\theta(E_L(x)) = E_L(\theta(x))$, so $G(L)$ is invariant under every automorphism of $L$. It follows from \cite[Corollary 3.2]{frat} that it is invariant under all derivations of $L$, whence the result.
\end{proof}
\bigskip

Next we collect together some basic properties of $\gamma(L)$; in particular, parts (iv) and (v) of the first lemma below generalise characterisations of nilpotent Lie algebras due to Barnes (see \cite{nilp} and \cite{barnes}).

\begin{lemma}\label{l:nilp} Let $L$ be a Lie algebra. Then
\begin{itemize}
\item[(i)] $\gamma(L)$ is nil on $L$, and so nilpotent;
\item[(ii)] $L$ is nilpotent if and only if $L = \gamma(L)$;
\item[(iii)] $L$ is nilpotent if and only if $\mathcal{G} = \emptyset$; 
\item[(iv)] $L$ is nilpotent if and only if $M$ is an ideal of $L$ for all $M \in \mathcal{G}$; and
\item[(v)] $L$ is nilpotent if and only if $L/B$ is nilpotent for some ideal $B$ of $L$ with $B \subseteq \gamma(L)$.
\end{itemize}
\end{lemma}
\begin{proof} (i) Let $x \in \gamma(L)$ and suppose that $E_L(x) \neq L$. Then there is a subalgebra $M \in \mathcal{G}$ such that $E_L(x) \subseteq M$. But $L = E_L(x) \oplus L_1(x)$ where $L_1(x) = \cap_{i=1}^{\infty} L (\hbox{ad}\,x)^i \subseteq \gamma(L)$, so $L = M$, a contradiction. The result follows.
\par
\noindent(ii) $L$ is nilpotent if and only if $E_L(x) = L$ for all $x \in L$, by Engel's Theorem. But this holds if and only if $\gamma(L) = L$.
\par
\noindent(iii) This is clear from (ii).
\par
\noindent(iv) If $L$ is nilpotent then all maximal subalgebras of $L$ are ideals of $L$ (see \cite{nilp}). Conversely, suppose that all maximal subalgebras of $L$ in $\mathcal{G}$ are ideals of $L$ and let $M \in \mathcal{G}$. Then $E_L(x) \subseteq M$ for some $x \in L$, so $I_L(M) = M$, by \cite[Corollary 4.4.4.4]{wint}. It follows that $\mathcal{G} = \emptyset$, and thus that $L$ is nilpotent, by (iii).
\par
\noindent(v) Suppose that $B$ is an ideal of $L$ with $B \subseteq \gamma(L)$ such that $L/B$ is nilpotent. Suppose that $\mathcal{G} \not = \emptyset$ and let $M \in \mathcal{G}$. Then there is an $x \in L$ such that $E_L(x) \subseteq M$. It follows from \cite[Corollary 4.4.4.4]{wint} that $I_L(M) = M$. Now $B \subseteq \gamma(L) \subseteq M$ and $M/B$ is a maximal subalgebra of $L/B$, so $M/B$ is an ideal of $L/B$. It follows that $M$ is an ideal of $L$, a contradiction. Hence $\mathcal{G} = \emptyset$ and $L$ is nilpotent, by (iii).
\par
The converse is clear.
\end{proof}

\begin{lemma} \label{l:eng} Let $L$ be a Lie algebra over a field $F$ and let $B$ be an ideal of $L$. Then 
\begin{itemize}
\item[(i)] $(E_L(x) + B)/B \subseteq E_{L/B}(x + B)$ for all $x \in L$;
\item[(ii)] if $B \subseteq M$ and $M/B \in \mathcal{G}$ then $M \in \mathcal{G}$; 
\item[(iii)] $(\gamma(L) + B)/B \subseteq \gamma(L/B)$; 
\item[(iv)] if $F$ has at least dim\,$L$ elements and $M \in \mathcal{G}$ then $M/B \in \mathcal{G}$ if $B \subseteq M$; and
\item[(v)] if $F$ has at least dim\,$L$ elements and $B \subseteq \gamma(L)$ then $\gamma(L)/B = \gamma(L/B)$.
\end{itemize}
\end{lemma}
\begin{proof} (i) Let $y \in E_L(x)$. Then $y$(ad$\,x)^n = 0$ for some $n \in \N$, and so $(y + B)$(ad$\,(x + B))^n \in B$. It follows that $y + B \in E_{L/B}(x + B)$.
\par

\noindent (ii) This follows from (i).
\par 

\noindent (iii) This follows from (ii).
\par

\noindent (iv) Suppose $M \in \mathcal{G}$. Then $M$ contains a Cartan subalgebra $C$ of $L$, by Lemma \ref{l:cart}. But now $(C + B)/B$ is a Cartan subalgebra of $L/B$, by \cite[Theorem 4.4.5.1]{wint}, whence $M/B \in \mathcal{G}$.
\par

\noindent (v) This follows from (iii) and (iv).
\end{proof}
\bigskip

The {\em Frattini subalgebra}, $F(L)$, of $L$ is the intersection of the maximal subalgebras of $L$, and the {\em Frattini ideal} is $\phi(L) = F(L)_L$. We say that $L$ is {\em $\phi$-free} if $\phi(L) = 0$. If $U$ is a subalgebra of $L$, the {\em idealiser} of $U$ in $L$ is $I_L(U) = \{x \in L : [x, U] \subseteq U\}$. Next we see that if $L$ is solvable then the elements of $\mathcal{G}$ are precisely the self-idealising maximal subalgebras of $L$.

\begin{propo}\label{p:solv} Let $L$ be solvable. Then $M \in \mathcal{G}$ if and only if $I_L(M) = M$.
\end{propo}
\begin{proof} If $M \in \mathcal{G}$ then $E_L(x) \subseteq M$ for some $x \in L$, whence $I_L(M) = M$, by \cite[Corollary 4.4.4.4]{wint}.
\par
Now let $L$ be a solvable Lie algebra of minimal dimension having a subalgebra $M$ such that $I_L(M) = M$ but $M \notin \mathcal{G}$. Then $L$ is $\phi$-free and $M_L = 0$, by Lemma \ref{l:eng} (ii). Let $A$ be a minimal ideal of $L$. Then $L = A \oplus M$. Pick a minimal ideal $B$ of $M$ and let $b \in B$. We have that $L = E_L(b) \oplus L_1(b)$. Moreover, $L_1(b) \subseteq A$, since $B$ is abelian, so $L = E_L(b) + A$. Now $A + B$ is an ideal of $L$, and hence so is $[A, B] = [A+B, A+B]$. If $E_L(b) = L$ for all $b \in B$ then $[A,B] = 0$ and $B$ is an ideal of $L$, contradicting the fact that $M_L = 0$. 
\par
It follows that there is a $b \in B$ and a maximal subalgebra $K$ of $L$ such that $E_L(b) \subseteq K$. Now $L = A \oplus K$ so $E_A(b) = 0$. But it is easy to check that $E_L(b) = E_A(b) \oplus E_M(b) = E_M(b)$, so $M \in \mathcal{G}$. This contradiction establishes the converse.
\end{proof} 

\begin{coro}\label{c:ss} Let $L$ be a solvable Lie algebra. Then $L$ is supersolvable if and only if $M$ has codimension one in $L$ for every $M \in \mathcal{G}$.
\end{coro}
\begin{proof} Suppose first that $M$ has codimension one in $L$ for every $M \in \mathcal{G}$. Let $K$ be any maximal subalgebra of $L$ such that $K \notin \mathcal{G}$. Then $K$ is an ideal of $L$, by Proposition \ref{p:solv}, and so has codimension one in $L$. It follows that $L$ is supersolvable, by \cite[Theorem 7]{barnes}. 
\par
The converse follows immediately from \cite[Theorem 7]{barnes}. 
\end{proof}

\begin{coro}\label{c:facss} Let $B$ be an ideal of $L$ with $B \subseteq \gamma(L)$ such that $L/B$ is supersolvable. Then $L$ is supersolvable.
\end{coro}
\bigskip

Following \cite{arch} we put
\[ \mathcal{T} = \{M : M \hbox{ is a maximal subalgebra of } L \hbox{ and } I_L(M) = M \}, \]
\[ T(L) = \bigcap_{M \in \mathcal{T}} M \hbox{ if } \mathcal{T} \hbox{ is non-empty}; \hspace{.2cm} T(L) = L \hbox{ otherwise};\hspace{.2cm} \tau(L) = T(L)_L. \]
Then we have the following result.

\begin{coro}\label{c:tau} If $L$ is solvable, then $\gamma(L) = \tau(L)$.
\end{coro}

The hypothesis that $L$ be solvable cannot be removed from Proposition \ref{p:solv}, as the next result will show. A subalgebra $U$ of a semisimple Lie algebra $L$ is called {\em regular} if we can choose a basis for $U$ in such a way that every vector of this basis is either a root vector of $L$ corresponding to some Cartan subalgebra $C$ of $L$, or otherwise belongs to $C$; $U$ is an {\em $R$-subalgebra} of $L$ if it is contained in a regular subalgebra of $L$, and is an {\em $S$-subalgebra} otherwise (see \cite[page 158]{dynk}).

\begin{propo}\label{p:reg} Let $L$ be a semisimple Lie algebra over an algebraically closed field $F$ of characteristic zero. Then $M \in \mathcal{G}$ if and only if $M$ is a regular maximal subalgebra of $L$.
\end{propo}
\begin{proof} Suppose first that $M \in \mathcal{G}$. Then $M$ contains a Cartan subalgebra of $L$ by Lemma \ref{l:cart} and so is clearly regular.
\par
Now suppose that $M$ is a regular maximal subalgebra of $L$. Then $M$ is either parabolic or semisimple of maximal rank (see \cite{dynk}), whence $M \in \mathcal{G}$.
\end{proof}
\bigskip

In view of the above result, if $M$ is a maximal subalgebra that is an $S$-subalgebra of a semisimple Lie algebra $L$ over an algebraically closed field of characteristic zero, then $M$ is self-idealising but not in $\mathcal{G}$. This observation yields the following corollary. 

\begin{coro}\label{c:amss} Let $L$ be a semisimple Lie algebra over an algebraically closed field $F$ of characteristic zero, and suppose that all maximal subalgebras of $L$ belong to $\mathcal{G}$. Then $L = sl(2,F)$.
\end{coro}
\begin{proof} First suppose that $L$ is simple. If $L \ne sl(2,F)$ then $L$ has an $S$-subalgebra isomorphic to $sl(2,F)$ (see \cite[section 9, pages 168 - 175]{dynk}). This subalgebra must be inside a maximal subalgebra which is itself an $S$-subalgebra, and so is not in $\mathcal{G}$.
\par

Now suppose that $L$ is semisimple, so that $L = S_1 \oplus \ldots \oplus S_n$, where $S_i$ is a simple ideal of $L$, and let $M$ be a maximal subalgebra of $S_i$. Then $$\sum^n_{\substack{j=1 \\ j \ne i}} S_i \oplus M$$ is a maximal subalgebra of $L$ and so contains a Cartan subalgebra $C$ of $L$. But now, if $\pi_i$ is the projection map of $L$ onto $S_i$, $\pi_i(C)$ is a Cartan subalgebra of $M$. It follows that $S_i \cong sl(2,F)$ for each $1 \leq i \leq n$. 
\par

Finally, let $S$ be $sl(2,F)$, let $\bar{S}$ be an isomorphic copy of $S$ and denote the image of $s \in S$ in $\bar{S}$ by $\bar{s}$. Put $L = S \oplus \bar{S}$ with $[S, \bar{S}] = 0$. It is easy to check that the diagonal subalgebra $M = \{ x \in L : x = s + \bar{s} \hbox{ for some } s \in S \}$ is maximal in $L$. However, $s \in E_L(s + \bar{s})$ for every $s \in S$, so $M \notin \mathcal{G}$.\end{proof} 
\bigskip

The {\em abelian socle} of $L$, $Asoc\,L$ is the sum of the minimal abelian ideals of $L$. We can use the above result to classify the Lie algebras all of whose maximal subalgebras of $L$ belong to $\mathcal{G}$. 

\begin{theor}\label{t:class} Let $L$ be a Lie algebra over an algebraically closed field $F$ of characteristic zero. Then the following are equivalent: 
\begin{itemize}
\item[(i)] all maximal subalgebras of $L$ belong to $\mathcal{G}$; and 
\item[(ii)] $L/ \phi(L) = Asoc\,(L/\phi(L)) \oplus S$ where $S \cong sl(2,F)$ and each minimal ideal of $L/\phi(L)$ has even dimension. 
\end{itemize}
\end{theor}
\begin{proof} Suppose first that (i) holds and that $\phi(L) = 0$. Then we have that $L = Asoc\,(L) \oplus (S \oplus B)$ where $S$ is a semisimple subalgebra, $B$ is an abelian subalgebra and $[S, B] = 0$, by \cite[Theorems 7.4 and 7.5]{frat}. Let $C$ be a Cartan subalgebra of $L$ and let $L = C \oplus L_1$ be the Fitting decomposition of $L$ relative to $C$. Clearly $L_1 \subseteq L^{(1)} \subseteq Asoc\,(L) \oplus S$. If $B \neq 0$ it follows that any maximal subalgebra of $L$ containing $Asoc\,(L) \oplus S$ cannot contain a Cartan subalgebra of $L$. This yields that $B = 0$. That $S \cong sl(2,F)$ then follows from Lemma \ref{l:eng}.
\par
Now let $A$ be a minimal ideal of $L$. Then $A$ is an irreducible $sl(2,F)$-module and so has the structure given in \cite[pages 83-86]{jac}. Let $Fh$ be a Cartan subalgebra of $S$ and let $D$ be its centralizer in $A$. Then $E = Fh + D$ is a Cartan subalgebra of $K = A + S$ (see \cite{dix}). If $A$ has odd dimension Cartan subalgebras of $K$ have dimension greater than 1 (put $i = m/2$ in \cite[(36), page 85]{jac}). But $S$ is a maximal subalgebra of $K$ and contains no such Cartan subalgebra.
\par
If $\phi(L) \neq 0$, then $L/\phi(L)$ satisfies the hypotheses of (i) by Lemma \ref{l:eng}, and so (ii) follows.
\par
So now assume that (ii) holds. We can assume that $\phi(L) = 0$, since, if we can prove (ii) for this case, the result will follow from Lemma \ref{l:eng}. Let $A = Asoc\,(L) = A_1 \oplus \ldots \oplus A_n$. It straightforward to check, using \cite[(36), page 85]{jac}), that if $Fh$ is a Cartan subalgebra of $S$ then $C_A(h) = 0$. It follows from \cite{dix} that the Cartan subalgebras of $L$ are one-dimensional. Now the maximal subalgebras of $L$ are of the form $A \oplus K$, where $K$ is a maximal subalgebra of $S$, or $$\sum^n_{\substack{j=1 \\ j \ne i}} A_i \oplus S.$$ Clearly each of these contains a Cartan subalgebra of $L$, and (i) follows.
\end{proof}

\section{Maximal subalgebras of codimension one}
Our objective here is to generalise Corollary \ref{c:ss} and \cite[Theorem 1]{codone}. First we recall the definition of the algebras $L_m(\Gamma)$ over a field $F$ of characteristic zero or $p$, where $p$ is prime, as given by Amayo in \cite[page 46]{am}. Let $m$ be a positive integer satisfying 
$$m = 1, \hspace{1cm} \hbox{or if } p \hbox{ is odd, } \hspace{.2cm} m = p^r - 2 \hspace{.2cm} (r \geq 1),$$ 
$$\hbox{or if } p =2, \hspace{.2cm} m = 2^r - 2 \hbox{ or } m = 2^r - 3 \hspace{.2cm} (r \geq 2).$$ 
Let $\Gamma = \{\gamma_0, \gamma_1, \ldots\} \subseteq F$ subject to 
$$(m + 1 - i) \gamma_i = \gamma_{m+i-1} = 0 \hspace{.3cm} \hbox{for all } i \geq 1, \hbox{ and}$$ $$\lambda_{i,k+1-i} \gamma_{k+1} = 0 \hspace{.3cm} \hbox{for all } i,k \hbox{ with } 1 \leq i \leq k.$$
Let $L_m(\Gamma)$ be the Lie algebra over $F$ with basis $v_{-1}, v_0, v_1, \ldots, v_m$ and products 
$$[v_{-1},v_i] = -[v_i,v_{-1}] = v_{i-1} + \gamma_i v_m, \hspace{1cm} [v_{-1}, v_{-1}] = 0,$$ 
$$[v_i,v_j] = \lambda_{ij} v_{i+j} \hbox{ for all } i,j \hbox{ with } 0 \leq i,j \leq m,$$
where $v_{m+1} = \ldots = v_{2m} = 0$.    
\par
We also let $H_{m,i}$ be the subspace spanned by $v_i, \ldots, v_m$.
\par
We shall need the following classification of Lie algebras with core-free subalgebras of codimension one as given in \cite{am}.

\begin{theor}\label{t:am} (\cite[Theorem 3.1]{am})
Let $L$ have a core-free subalgebra of codimension one. Then either (i) dim $L \leq 2$, or else (ii) $L \cong L_m(\Gamma)$ for some $m$ and $\Gamma$ satisfying the above conditions.
\end{theor}
\medskip

We shall also need the following properties of $L_m(\Gamma)$ which are given by Amayo in \cite{am}.

\begin{theor}\label{t:gamma} (\cite[Theorem 3.2]{am})
\begin{itemize}
\item[(i)] If $m > 1$ and $m$ is odd, then $L_m(\Gamma)$ is simple and $H_{m,0}$ is the only subalgebra of codimension one.
\item[(ii)] If $m > 1$ and $m$ is even, then $L_m(\Gamma)$ has precisely two subalgebras of codimension one in $L_m$, namely $L_m^{(1)}$ and $H_{m,0}$.
\item[(iii)] $L_1(\Gamma)$ has a basis $\{u_{-1}, u_0, u_1 \}$ with multiplication $[u_{-1}, u_0] = u_{-1} + \gamma_0 u_1$ $(\gamma_0 \in F, \gamma_0 = 0$ if $\Gamma = \{0\})$, $[u_{-1}, u_1] = u_0, [u_0, u_1] = u_1$.
\item[(iv)] If $F$ has characteristic different from two then $L_1(\Gamma) \cong L_1(0) \cong sl(2,F)$.
\item[(v)] If $F$ has characteristic two then $L_1(\Gamma) \cong L_1(0)$ if and only if $\gamma_0$ is a square in $F$. 
\end{itemize}
\end{theor}

\begin{lemma}\label{l:dim} Let $x \in L$ be ad-nilpotent. Then (ad\,$x)^{dim\,L} = 0$.
\end{lemma}
\begin{proof} It is easy to check that if $y\,(\hbox{ad\,}x)^{n} = 0$ but $y\,(\hbox{ad\,}x)^{n-1} \neq 0$ then the elements $y, y\,(\hbox{ad\,}x), \ldots, y\,(\hbox{ad\,}x)^{n-1}$ are linearly independent.
\end{proof}

\begin{lemma}\label{l:simple} Let $L \cong L_m(\Gamma)$. Then all maximal subalgebras $M \in {\mathcal G}$ have codimension one in $L$ if and only if $L \cong L_1(0)$.
\end{lemma}
\begin{proof} Let $L = L_m(\Gamma)$ have all maximal subalgebras $M \in {\mathcal G}$ with codimension one in $L$, and suppose that $m > 1$. We have that $v_{-1} \not \in H_{m,0}$, and it is shown in the proof of \cite[Theorem 3.2]{am} that $v_m \not \in L_m^{(1)}$, so $v_{-1} + v_m$ does not belong to a subalgebra of codimension one in $L$. It follows that $E_L(v_{-1} + v_m) = L$ and $v_{-1} + v_m$ is ad-nilpotent. Now,
$$ [v_i, v_{-1} + v_m] = - v_{i-1} - \gamma_i v_m \hspace{1cm} (i \geq 1);$$ $$[v_0, v_{-1} + v_m] = - v_{-1} - (\gamma_0 + \lambda_{m0})v_m ;[v_{-1}, v_{-1} + v_m] = - v_{m-1} - \gamma_m v_m.$$
This yields that $v_i(\hbox{ad\,}(v_{-1} + v_m))^{m+2}$ is $(-1)^{m+2}\gamma_i v_{-1} + f(v_0, v_1, \ldots v_m)$ for $1 \leq i \leq m$. Since $m+2 = \hbox{dim\,}L$ it follows from Lemma \ref{l:dim} that $\gamma_i = 0$ for $1 \leq i \leq m$. But $v_{-1} + v_0 + v_m$ is also ad-nilpotent and a similar calculation shows that $v_1(\hbox{ad\,}(v_{-1} + v_0 + v_m))^{m+2}$ is $(-1)^{m+2}v_0 + g(v_{-1}, v_1, \ldots v_m)$, which contradicts Lemma \ref{l:dim}. Hence $m = 1$.
\par
Suppose that $L \not \cong L_1(0)$. Then Theorem \ref{t:gamma} (iii),(iv) implies that $F$ has characteristic two and $L$ has a basis $u_{-1}, u_0, u_1$ with multiplication $[u_{-1}, u_0] = u_{-1} + \gamma_0 u_1$, $[u_{-1}, u_1] = u_0, [u_0, u_1] = u_1$ where $\gamma_0$ is not a square in $F$. But a simple calculation verifies that $Fu_{-1}$ is a maximal subalgebra of $L$ containing $E_L(u_{-1}) = Fu_{-1}$, contradicting our hypothesis. It follows that $L \cong L_1(0)$.
\par
The converse follows from \cite[Theorem 1]{codone}.
\end{proof}
\bigskip

\begin{theor}\label{t:main} Let $L$ be a Lie algebra over a field with at least dim\,$L$ elements. Then the following are equivalent:
\begin{itemize}
\item[(i)] every maximal subalgebra $M \in {\mathcal G}$ has codimension one in $L$; and \item[(ii)] $L/\gamma(L) = S \oplus R$, where $S = S_1 \oplus \ldots \oplus S_n$, $S_i$ is a simple ideal of $L/\gamma(L)$ isomorphic to $L_1(0)$ for each $1 \leq i \leq n$, or is $\{0\}$, and $R$ is a supersolvable ideal of $L/\gamma(L)$ (possibly $\{0\}$).
\end{itemize}
\end{theor}
\begin{proof} ((i) $\Rightarrow$ (ii)): Suppose that (i) holds and assume first that $\gamma(L) = 0$. If $L$ is solvable it is supersolvable, by Corollary \ref{c:ss}, so suppose that $L$ is not solvable. Clearly $L$ has a maximal subalgebra $M_1 \in {\mathcal G}$ (otherwise, $L$ is nilpotent). If all $M \in {\mathcal G}$ have $L/M_L$ solvable, then $L^{(2)} \subseteq \gamma(L) = 0$, and $L$ is solvable, so we can further assume that $L/(M_1)_L \cong L_1(0)$, by Theorem \ref{t:am} and Lemma \ref{l:simple}. If $(M_1)_L = 0$ we have finished, so suppose that $(M_1)_L \neq 0$. 
\par
Let $M_2 \in {\mathcal G}$ be a maximal subalgebra with $(M_1)_L \not \subseteq M_2$. Then $L = (M_1)_L + M_2$. Put $B = (M_1)_L + (M_2)_L$. Since $L/(M_1)_L$ is simple, $B = (M_1)_L$ or $B = L$. The former implies that $(M_1)_L = (M_2)_L = M_2$, a contradiction, so $L = B = (M_1)_L + (M_2)_L$. Now $L/((M_1)_L \cap (M_2)_L) \cong (L/(M_1)_L) \oplus (L/(M_2)_L)$. Suppose there is such an $M_2$ with $L/(M_2)_L \cong L_1(0)$. If $(M_1)_L \cap (M_2)_L = 0$ we have finished. If $(M_1)_L \cap M_2)_L \neq 0$ then choose $M_3 \in {\mathcal G}$ to be a maximal subalgebra with $(M_1)_L \cap (M_2)_L  \not \subseteq M_3$. In similar fashion to that above we find that $L = ((M_1)_L \cap (M_2)_L) + (M_3)_L$. If there is such an $M_3$ with $L/(M_3)_L \cong L_1(0)$ and $(M_1)_L \cap (M_2)_L \cap (M_3)_L \neq 0$ we continue in the same way.
\par
Eventually obtain $L = A + (M_n)_L$, where $A = (M_1)_L \cap \ldots \cap (M_{n-1})_L$, $M_n \in {\mathcal G}$ is a maximal subalgebra of $L$ with $A \not \subseteq M_n$, $L/(A \cap (M_n)_L) \cong (L/(M_n)_L) \oplus S_1 \oplus \ldots \oplus S_{n-1}$, each $S_i \cong L/(M_i)_L \cong L_1(0)$ for $1 \leq i \leq n-1$ and either $L/(M_n)_L \cong L_1(0)$ and $A \cap (M_n)_L = 0$, in which case we have finished, or else there is no $M_n \in {\mathcal G}$ with $L/(M_n)_L \cong L_1(0)$ and dim\,$(L/(M_n)_L) \leq 2$. So suppose the latter holds, in which case $L^{(2)} \subseteq (M_n)_L$. We now have that for every $M \in \mathcal{G}$, either $A \subseteq M$ or $L^{(2)} \subseteq M$, which yields that $A \cap L^{(2)} \subseteq \gamma(L) = 0$, whence $A^{(2)} = 0$ and $A$ is solvable.  
\par
Now $L/A \cong (L/(A \cap (M_n)_L))/(A/(A \cap (M_n)_L)) \cong S_1 \oplus \ldots \oplus S_{n-1}$ (since $A/(A \cap (M_n)_L \cong L/(M_n)_L$), so $L = L^{(2)} \oplus A$. Moreover, $A$ is supersolvable, by Corollary \ref{c:ss}, and $L^{(2)} \cong L/A$, which completes the proof for the case $\gamma(L) = 0$.
\par
Now suppose that $\gamma(L) \neq 0$. Then if $M/\gamma(L) \in \mathcal{G}$ we have $M \in \mathcal{G}$, by Lemma \ref{l:eng}(ii), and so $M/\gamma(L)$ has codimension one in $L/\gamma(L)$. Thus (ii) follows from above.
\medskip

\noindent ((ii) $\Rightarrow$ (i)): So now suppose that (ii) holds and let $M \in \mathcal{G}$. Then $L/\gamma(L)$ is $\phi$-free, so $\overline{L} = L/\gamma(L) = S \oplus R = S \oplus (A + B)$, where $A = A_1 \oplus \ldots \oplus A_k = Asoc\, L$ and $B$ is abelian, by \cite[Theorems 7.3 and 7.4]{frat}. Since $R = A + B$ is supersolvable, dim\,$A_i = 1$ for each $1 \leq i \leq k$. Clearly $\overline{M} = M/\gamma(L)$ is a maximal subalgebra of $\overline{L}$. If $A \not \subseteq \overline{M}$ then there is an $A_i \not \subseteq \overline{M}$ for some $1 \leq i \leq k$. But then $\overline{L} = A_i + \overline{M}$ and $\overline{M}$ has codimension one in $\overline{L}$, whence $M$ has codimension one in $L$. So assume that $A \subseteq \overline{M}$.
\par
Suppose that $B \not \subseteq \overline{M}$.Then there is an element $b \in B$ such that $b \not \in \overline{M}$. But $[B, \overline{L}] \subseteq A \subseteq \overline{M}$, so $\overline{L} = \overline{M} + Fb$, and again $M$ has codimension one in $L$.
\par
So suppose that $R \subseteq \overline{M}$. Suppose further that there exist $i, j$ with $1 \leq i, j \leq n$ and $S_i \not \subseteq \overline{M}$, $S_j \not \subseteq \overline{M}$. Then $\overline{L} = \overline{M} + S_i = \overline{M} + S_j$. Moreover, $[\overline{L}, \overline{M} \cap S_i] = [\overline{M} + S_j, \overline{M} \cap S_i] = [\overline{M}, \overline{M} \cap S_i] \subseteq \overline{M} \cap S_i$, so $\overline{M} \cap S_i$ is an ideal of $\overline{L}$. It follows that $\overline{M} \cap S_i = 0$. Since $\overline{M} \in {\mathcal G}$, by Lemma \ref{l:eng}(iv), $E_L(r + \sum_{t=1}^n x_t) \subseteq \overline{M}$ for some $r \in R, x_t \in S_t$. But now $x_i \in E_L(r + \sum_{t=1}^n x_t) \cap S_i = 0$ for each $x_i \in S_i$, which yields that $S_i \subseteq E_L(r + \sum_{t=1}^n x_t) \cap S_i = 0$, a contradiction. We therefore have that there is just one $S_i$ with $S_i \not \subseteq \overline{M}$, in which case 
$$\overline{M} = R \oplus \sum^n_{\substack{j=1 \\ j \ne i}} S_j \oplus K,$$
where $K$ is a maximal subalgebra of $S_i$. It follows from Lemma \ref{l:simple} that $M$ has codimension one in $L$.  
\end{proof}
\bigskip

Then we have the following corollary to Theorem \ref{t:main}
\bigskip

\begin{coro}\label{c:ssd} Let $L$ be a Lie algebra over a field with at least dim\,$L$ elements. Then $L$ is supersolvable if and only if $\delta(L) = L$ and $L/\gamma(L)$ has no ideals isomorphic to $L_1(0)$.
\end{coro}
\bigskip

Next we have some basic properties of $\eta(L)$ and $\delta(L)$.
\bigskip 

\begin{lemma} \label{l:fac} Let $L$ be a Lie algebra over a field $F$ and let $B$ be an ideal of $L$. Then 
\begin{itemize}
\item[(i)] $(\eta(L) + B)/B \subseteq \eta(L/B)$ and $(\delta(L) + B)/B \subseteq \delta(L/B)$; 
\item[(ii)] if $B \subseteq \eta(L)$ then $\eta(L)/B = \eta(L/B)$;
\item[(iii)] if $F$ has at least dim\,$L$ elements and $B \subseteq \delta(L)$ then $\delta(L)/B = \delta(L/B)$.
\end{itemize}
\end{lemma}
\begin{proof} This is straightforward.
\end{proof}
\bigskip

Define the the series $\{Z_i : i \geq 0\}$ inductively by $Z_0 = \{0\}, Z_i/Z_{i-1} = Z(L/Z_{i-1})$ for all $i \geq 1$, where $Z(L)$ is the centre of $L$. Then the {\em hypercentre} of $L$ is $Z_{\infty} = \cup_{i=0}^{\infty} Z_i$.
\bigskip

\begin{propo}\label{p:eta} For any Lie algebra $L$, $Z_{\infty} \subseteq \eta(L) \subseteq \delta(L)$.
\end{propo}
\begin{proof} Suppose that $Z_{\infty} \not \subseteq \eta(L)$. Then there is a maximal subalgebra $M \in \mathcal{H}$ and $k \geq 1$ such that $Z_k \not \subseteq M$ whereas $Z_{k-1} \subseteq M$. But now $L = M + Z_k$, from which it follows that $M$ is an ideal of $L$ and so has codimension one in $L$, a contradiction. The other inclusion is clear.
\end{proof}

\begin{propo}\label{p:delta} For any solvable Lie algebra $L$, $\delta(L)$ (and hence also $\eta(L)$) is supersolvable.
\end{propo}
\begin{proof} Let $L$ be a minimal counter-example. Suppose first that $\gamma(L) \neq 0$. Then $\delta(L)/\gamma(L) \subseteq \delta(L/\gamma(L))$ is supersolvable. But 
$$\frac{\delta(L)}{\gamma(L)} \cong \frac{(\delta(L)/\phi(L))}{(\gamma(L)/\phi(L))} \cong \frac{(\delta(L)/\phi(L))}{(\tau(L)/\phi(L))} = \frac{(\delta(L)/\phi(L))}{Z(L/\phi(L))},$$ 
by Corollary \ref{c:tau} and \cite[Theorem 2.8]{arch}. It follows that $\delta(L)/\phi(L)$ is supersolvable, and hence that $\delta(L)$ is supersolvable, by \cite[Theorem 6]{barnes}. 
\par
So suppose now that $\gamma(L) = 0$. Let $A$ be a minimal ideal of $L$ with $A \subseteq \delta(L)$. Then there is a maximal subalgebra $M \in \mathcal{G}$ of $L$ with $L = A \oplus M$. If dim\,$A > 1$ then $A \subseteq \delta(L) \subseteq M$, a contradiction. Hence dim\,$A = 1$. But $\delta(L)/A \subseteq \delta(L/A)$ is supersolvable, whence $\delta(L)$ is supersolvable. This contradiction completes the proof. 
\end{proof}
\bigskip

A subalgebra $U$ of $L$ is called a {\em c-ideal} of $L$ if there is an ideal $C$ of $L$ such that $L = U + C$ and $U \cap C \leq U_L$. Finally we have the following generalisation of \cite[Theorem 3.1]{cideal}.
\bigskip

\begin{theor}\label{t:cideal} Let $L$ be a Lie algebra over any field $F$. Then all maximal subalgebras $M$ of $L$ with $M \in \mathcal{D}$ are c-ideals of $L$ if and only if $L$ is solvable.
\end{theor}
\begin{proof} This follows closely that of \cite[Theorem 3.1]{cideal}, but we include the details for the convenience of the reader. Let $L$ be a non-solvable Lie algebra of smallest dimension in which all maximal subalgebras $M$ of $L$ with $M \in \mathcal{D}$ are c-ideals of $L$. Clearly all maximal subalgebras $M$ of $L$ with $M \in \mathcal{G}$ are c-ideals of $L$ and $\mathcal{G} \neq \emptyset$. Then all proper factor algebras of $L$ are solvable, by \cite[Lemma 2.1 (ii)]{cideal} and Lemma \ref{l:eng}. Suppose first that $L$ is simple. Let $M$ be a maximal subalgebra of $L$ with $M \in \mathcal{G}$. Then $M$ is a c-ideal so there is an ideal $C$ of $L$ such that $L = M + C$ and $M \cap C \leq M_L = 0$, as $L$ is simple. This yields that $C$ is a non-trivial proper ideal of $L$, a contradiction. If $L$ has two minimal ideals $B_1$ and $B_2$, then $L/B_1$ and $L/B_2$ are solvable and $B_1 \cap B_2 = 0$, so $L$ is solvable. Hence $L$ has a unique minimal ideal $B$ and $L/B$ is solvable.
\par
Suppose there is an element $b \in B$ such that ${\rm ad}_L b$ is not nilpotent. Let $L = E_L(b) \oplus L_1$ be the Fitting decomposition of $L$ relative to ${\rm ad}_L b$. Then $L \neq E_L(b)$ so let $M$ be a maximal subalgebra of $L$ containing $E_L(b)$. As $M \in \mathcal{G}$, it is a c-ideal and so there is an ideal $C$ of $L$ such that $L = M + C$ and $M \cap C \leq M_L$. Now $L_1 \leq B$ so $B \not \leq M_L$. It follows that $M_L = 0$ whence $M = E_L(b)$ and $B = C = L_1$. But $b \in M \cap B = 0$. Hence every element of $B$ is ad-nilpotent, yielding that $B$ is nilpotent and so $L$ is solvable, a contradiction. 
\par
The converse follows from \cite[Theorem 3.1]{cideal}.
\end{proof}

\end{document}